\documentclass[11pt]{article}
\usepackage{epigamath}

\usepackage[notext]{kpfonts}
\usepackage{baskervald}

\setpapertype{A4}

\usepackage[english]{babel}
\usepackage[all,cmtip]{xy}  
\usepackage{graphicx}
\usepackage{xcolor}
\usepackage{enumerate} 

\newtheorem{theorem}{Theorem}[section]
\newtheorem{question}[theorem]{Question}

\newtheorem{corollary}[theorem]{Corollary}

\newtheorem{definition}[theorem]{Definition}

\newtheorem{lemma}[theorem]{Lemma}

\newtheorem{proposition}[theorem]{Proposition}

\newcommand{\Z}{\mathbb Z}
\newcommand{\Q}{\mathbb Q}

\newcommand{\C}{\mathbb C}

\newcommand{\CP}{\mathbb P}

\newcommand{\LL}{\mathbb L}

\newcommand{\im}{\operatorname{im}}

\newcommand{\id}{\operatorname{id}}

\newcommand{\Spec}{\operatorname{Spec}}

\newcommand{\pr}{\operatorname{pr}}

\newcommand{\CH}{\operatorname{CH}}

\newcommand{\F}{\mathbb F}

\DeclareMathOperator{\spe}{sp}
 
\newcommand{\dashedlongrightarrow}{\xymatrix@1@=15pt{\ar@{-->}[r]&}}
\renewcommand{\longrightarrow}{\xymatrix@1@=15pt{\ar[r]&}}
\renewcommand{\mapsto}{\xymatrix@1@=15pt{\ar@{|->}[r]&}}
\renewcommand{\twoheadrightarrow}{\xymatrix@1@=15pt{\ar@{->>}[r]&}}
\newcommand{\hooklongrightarrow}{\xymatrix@1@=15pt{\ar@{^(->}[r]&}}
\newcommand{\congpf}{\xymatrix@1@=15pt{\ar[r]^-\sim&}}
\renewcommand{\cong}{\simeq}

\removelength{1cm}

\title{Variation of stable birational types in positive characteristic}
\titlemark{Variation of stable birational types}
\author{Stefan Schreieder}
\authoraddresses{
\authordata{Stefan Schreieder}{\firstname{Stefan} \lastname{Schreieder}\\
\institution{Mathematisches Institut, LMU M\"unchen, Theresienstr.\ 39, 80333 M\"unchen, Germany.}\\
\email{schreieder@math.lmu.de}}
}
\authormark{S. Schreieder}
\date{\vspace{-5ex}} 
\journal{\'Epijournal de G\'eom\'etrie Alg\'ebrique} 
\acceptation{Received by the Editors on August 30, 2019, and in final form on November 27, 2019. \\ Accepted on December 29, 2019.}

\newcommand{\articlereference}{Volume 3 (2019), Article Nr. 20}

\usepackage{float}

\numberwithin{equation}{section}

\makeatletter

\renewcommand{\p@equation}{\arabic{section}.\arabic{equation}\expandafter\@gobble}
\makeatother

\begin{document}


\maketitle



\begin{prelims}

\vspace{-0.55cm}

\def\abstractname{Abstract}
\abstract{Let $k$ be an uncountable algebraically closed field and let $Y$ be a smooth projective $k$-variety which does not admit a decomposition of the diagonal. We prove that $Y$ is not stably birational to a very general hypersurface of any given degree and dimension. We use this to study the variation of the stable birational types of Fano hypersurfaces over fields of arbitrary characteristic. This had been initiated by Shinder \cite{Sh19}, whose method works in characteristic zero.}




\languagesection{Fran\c{c}ais}{%

\vspace{-0.05cm}
{\bf Titre. Variation du type birationnel stable en caract\'eristique positive}
\commentskip {\bf R\'esum\'e.} Soient $k$ un corps alg\'ebriquement clos non d\'enombrable et $Y$ une $k$-vari\'et\'e lisse et projective qui n'admet pas de d\'ecomposition de la diagonale. Nous montrons que $Y$ n'est pas stablement birationnelle \`a une hypersuface tr\`es g\'en\'erale de degr\'e et de dimension quelconques. Nous mettons ceci \`a profit pour \'etudier la variation du type birationnel stable des hypersurfaces de Fano sur des corps de caract\'eristique arbitraire. Ceci a \'et\'e initi\'e par Shinder \cite{Sh19}, dont la m\'ethode fonctionne en caract\'eristique nulle.}

\end{prelims}


\newpage

\setcounter{tocdepth}{1} \tableofcontents

\section{Introduction} 

Starting with a breakthrough of Voisin \cite{Voi15}, an improvement by Colliot-Th\'el\`ene and Pirutka \cite{CTP16}, and  later a further improvement by the author \cite{Sch19}, there has recently been major progress in proving that certain unirational or rationally connected varieties are not stably rational, see e.g.\ \cite{AO,bea,BB,CTP16,CTP2,HKT,HPT,HPT2,HT,KO,oka,Sch2,Sch19,Sch18,totaro,Voi15}.
In fact, in all these examples, it is shown that the variety in question does not admit a decomposition of the diagonal, which implies stable irrationality, because stably rational varieties admit such a decomposition.

Motivated by a recent result of Shinder \cite{Sh19} with an appendix by Voisin \cite{Voi19}, we prove in this paper that the non-existence of a decomposition of the diagonal does not only prevent the variety from being stably rational, but in fact from being stably birational to a very general hypersurface of any given degree and dimension.

\begin{theorem} \label{thm:hypersurface:2} 
Let $X$ and $Y$ be smooth projective varieties over an uncountable algebraically closed field $k$ of arbitrary characteristic, and let $d,n\geq 1$ be integers. 
Assume that $Y$ does not admit a decomposition of the diagonal and that $X\subset \CP^{n+1}_k$ is a hypersurface of degree $d$  which is very general with respect to $Y$.

Then $X$ and $Y$ are not stably birational to each other. 
\end{theorem}

Note that $X$ and $Y$ in the above theorem are not necessarily of the same dimension, and $Y$ is not assumed to be a hypersurface.
The condition that the hypersurface $X\subset \CP^{n+1}_k$ is very general with respect to $Y$ means that it lies outside a countable union of proper closed subsets of the linear series $|\mathcal O_{\CP^{n+1}_k}(d)|$, and these subsets depend on $Y$. 
For $d=1$, our theorem specializes to the aforementioned well-known fact that a smooth projective variety which does not admit a decomposition of the diagonal is stably irrational.

Theorem \ref{thm:hypersurface:2} will be deduced from a more general result, which we prove in Theorem \ref{thm:degeneration:mixed:char:1} below and which applies to a wide range of varieties other than hypersurfaces as well.
In fact,  it  applies to any variety which admits a strictly semi-stable degeneration whose special fibre has universally trivial Chow group of zero-cycles.
In the case of hypersurfaces, we  use a degeneration to a general hyperplane arrangement as in \cite{Sh19}.

The above theorem implies in particular the following.

\begin{corollary} \label{cor:hypersurface}
Let $k$ be an uncountable algebraically closed field of arbitrary characteristic.
If there is a hypersurface $Y\subset \CP^{n+1}_k$ of degree $d$ which does not admit a decomposition of the diagonal, then two very general hypersurfaces of degree $d$ in $\CP^{n+1}_k$ are not stably birational to each other.
\end{corollary}

In characteristic zero, a slightly stronger variant of Corollary \ref{cor:hypersurface} had previously been proven by Shinder \cite[Theorem 1.2]{Sh19}, who obtains the same conclusion under the assumption that $Y$ is smooth and stably irrational. 

Improving earlier results of Koll\'ar \cite{kollar} 
 and Totaro \cite{totaro}, the author showed in \cite{Sch18} that over any uncountable field $k$ of characteristic different from two, a very general hypersurface $X\subset\CP_{k}^{n+1}$ of dimension $n\geq 3$ and degree $d\geq \log_2(n)+2$ does not admit a decomposition of the diagonal and so it is not stably rational.  
By the above corollary, we thus obtain the following strengthening of that result.

\begin{corollary} \label{cor:hypersurface:char_neq_2}
Let $k$ be an uncountable field of characteristic different from two.
Then two very general hypersurfaces in $\CP^{n+1}_k$ of degree $d\geq \log_2(n)+2$ and dimension $n\geq 3$ are not stably birational to each other.
\end{corollary}

By  \cite[Corollary 1.2]{Sch-torsion}, which generalizes the main result of \cite{Sch18} to characteristic two,  we obtain the following variant in characteristic two.

\begin{corollary} \label{cor:hypersurface:char=2}
Let $k$ be an uncountable field of characteristic  two.
Then two very general hypersurfaces in $\CP^{n+1}_k$ of degree $d\geq  \log_2 (n)+3$ are not stably birational to each other.
\end{corollary}

In \cite{Sh19}, Shinder  
deduced Corollary \ref{cor:hypersurface:char_neq_2}  
in the case where $\operatorname{char}(k)=0$ from \cite{Sch19} with the help of his joint work with Nicaise  \cite{NS}  on the specialization of stable birational types in characteristic zero.
In an appendix to \cite{Sh19}, Voisin  \cite{Voi19} used decompositions of the diagonal and unramified cohomology to obtain similar results in low dimensions.

Since \cite{NS}, as well as Kontsevich--Tschinkel's generalization in \cite{KT}, rely heavily on the weak factorization theorem, and hence on resolution of singularities, it is unclear how to generalize Shinder's approach to positive characteristic.
Similarly, Voisin's analysis \cite{Voi19} requires resolution of singularities of the singular hypersurfaces with unramified cohomology constructed in \cite{Sch18}, which in general is unknown in positive characteristic.
Moreover, Voisin's approach did not allow to reprove Shinder's characteristic zero result in dimensions greater than nine, because it requires the knowledge of certain unramified cohomology groups of very general hypersurfaces, which seems out of reach in high dimensions.

Our usage of decompositions of the diagonal is more direct than in \cite{Voi19} and we do in particular not use unramified cohomology.
Instead, our approach relies on a moving lemma of Gabber, Liu and Lorenzini \cite[Theorem 2.3]{GLL} and intersection theory on strictly semi-stable schemes using Fulton's work \cite{fulton}. 
While we use a degeneration to a general hyperplane arrangement as in \cite{Sh19}, our approach does not rely on the weak factorization theorem or Hironaka's resolution of singularities and does in particular not use the results in \cite{NS} and \cite{KT}. 
 
While the above results are formulated over uncountable fields,  Theorem \ref{thm:degeneration:mixed:char:1} below together with \cite{totaro} and \cite{Sch18} also allows to produce explicit examples of Fano hypersurfaces $X$ and $X'$  over small fields $k$ (e.g.\ $k=\Q(t)$ or $k=\F_p(t,s)$) which over the algebraic closure of $k$ are neither stably rational, nor stably birational to each other.
If $k$ is not of characteristic two, the slopes $\frac{\deg X}{\dim X+1}$ of these examples may be chosen arbitrarily small, see Theorem \ref{thm:ex:Q(t)} below.

It is worth to compare the results on variation of stable birational types in \cite{Sh19} and the present paper with the concept of birational rigidity (see \cite{Kol18} for a recent survey), which allows to prove that for certain classes of smooth projective Fano varieties of Picard rank one, any birational equivalence is an isomorphism.
In particular, in these cases the birational types vary as much as possible.  
By a result of de Fernex \cite{deF},  
this applies for instance to smooth projective hypersurfaces $X\subset \CP^{n+1}_\C$ of degree $d=n+1$ and dimension $n\geq 3$. 
The corresponding result in positive characteristic is open.
Birational rigidity has so far mostly been applied to Fano varieties of index one and two (see e.g.\ \cite{Pu16,Pu18} for the index two case) and it is unknown whether the method applies to hypersurfaces $X\subset \CP^{n+1}_\C$ of degree $d\ll n$.
Finally, the condition on the Picard rank seems to prevent applications to questions about stable birational equivalence.

\section{Notations}

All schemes are separated.
An algebraic scheme is a scheme of finite type over a field.
A variety is an integral algebraic scheme. 
If $k$ is an uncountable field, a very general point of a $k$-variety $X$ is a closed point outside a countable union of proper closed subsets.

Let $R$ be a ring and let $X$ be an $R$-scheme and $f$ be a morphism of $R$-schemes.
Then for any ring extension $R\subset A$, we denote the base change of $X$ and $f$ to $A$ by $X_A$ and $f_A$, respectively.

We denote the Chow group of algebraic cycles of dimension $r$, resp.\ codimension $r$, modulo rational equivalence on an algebraic scheme $X$ by $\CH_r(X)$, resp.\ $\CH^r(X)$.
For a standard reference on Chow groups, see \cite{fulton}.
For a codimension $r$ cycle $\gamma \in Z^r(X)$, we denote by $[\gamma]\in \CH^r(X)$ its class in the Chow group and by $|\gamma|$ its support, which is a reduced closed subscheme of pure codimension $r$ in $X$.
We say that two cycles $\gamma\in Z^a(X)$ and $\gamma'\in Z^b(X)$ meet properly if $|\gamma|\cap |\gamma'|$ has the expected codimension $a+b$ in $X$.
We use the same notation for Chow groups and cycles in the more general setting where $X$ is only a scheme of finite type over a regular base scheme $S$, replacing the dimension and codimension by an appropriate relative notion, see \cite[Chapter 20]{fulton}.

An algebraic scheme $X$ of pure dimension $n$ over a field $k$ admits a decomposition of the diagonal if 
$$
[\Delta_X]=[X\times z]+[Z_X]\in \CH_n(X\times_k X) ,
$$
where $z$ is a zero-cycle on $X$ and $Z_X$ is a cycle on $X\times_k X$ which does not dominate the first factor.
If $X$ is integral with function field $K=k(X)$, then this is equivalent to 
$$
[\delta_X]=[z_{K}]\in \CH_0(X_{K}), 
$$
where $z$ is the zero-cycle on $X$ from above 
and $\delta_X$ denotes the zero-cycle on $X_K$ that is induced by the diagonal $\Delta_X$.
If $X$ is a proper variety that is stably rational (\emph{i.e.}\ $X\times \CP^m$ is rational for some $m\geq 0$), then it admits a decomposition of the diagonal, see \cite[Lemme 1.5]{CTP16} for the case where $X$ is smooth and \cite[Lemma 2.4]{Sch18} in general.

We say that a proper algebraic scheme $X$ over a field $k$ has universally trivial Chow group of zero-cycles, if the degree map $\deg:\CH_0(X_L)\to \Z$ is an isomorphism for any field extension $L/k$.
If $X$ is geometrically integral and smooth over $k$, then this is equivalent to the fact that $X$ admits a decomposition of the diagonal, see e.g.\ \cite[Proposition 1.4]{CTP16}.


Let $R$ be a discrete valuation ring with residue field $k$ and fraction field $K$.
For an $R$-scheme $\mathcal X$, we denote by $X_0:=\mathcal X\times_R k$ the special fibre and by $X_\eta:=\mathcal X\times_R K$ and $X_{\overline \eta}:=\mathcal X\times_R \overline K$ the generic and the geometric generic fibres of $\mathcal X\to \Spec R$, respectively. 
If $\mathcal X$ is proper and flat over $R$, we say that $X_\eta$ degenerates or specializes to $X_0$.
If additionally $k$ is algebraically closed, we also say that for any field extension $L/K$, the base change $X_\eta\times_K L$ (e.g.\ $X_{\overline \eta}$) degenerates or specializes to $X_0$, cf.\ \cite[Section 2.2]{Sch19}.

We recall the definition of strictly semi-stable $R$-schemes, which will be particularly important for us, see e.g.\ \cite[Definition 1.1]{hartl}.

\begin{definition}
Let $R$ be a discrete valuation ring and let $\pi:\mathcal X\to \Spec R$ be a proper flat morphism with $\mathcal X$ integral.
The $R$-scheme $\mathcal X$ (or the morphism $\pi$) is called strictly semi-stable, if the special fibre $X_0$ is a geometrically reduced simple normal crossing divisor on $\mathcal X$, \emph{i.e.}\ the irreducible components of $X_0$ are smooth Cartier divisors on $\mathcal X$ and the scheme-theoretic intersection of $r$ different components of $X_0$ is either empty or smooth and equi-dimensional of codimension $r$ in $\mathcal X$.
\end{definition}

The generic fibre $X_\eta$ of a strictly semi-stable $R$-scheme is automatically smooth, because we assumed $\pi$ to be proper in the above definition.
Moreover, the total space $\mathcal X$ is regular, because its generic fibre is smooth and each component of the special fibre is a smooth Cartier divisor, see \cite[Remarks 1.1.1 and 1.1.2]{hartl}.

The following lemma is an immediate consequence of  \cite[Proposition 1.3]{hartl}.

\begin{lemma} \label{lem:ss}
Let $R$ be a discrete valuation ring and let $\mathcal X$ be a strictly semi-stable $R$-scheme and $\mathcal X'$ be a smooth and proper $R$-scheme.
Then, $\mathcal X\times_R\mathcal X'$ is also a strictly semi-stable $R$-scheme.
Moreover, if $A$ denotes the local ring of $\mathcal X'$ at a generic point of $X'_0$, then the base change $\mathcal X_A:=\mathcal X\times_R \Spec A$ is a strictly semi-stable $A$-scheme.
\end{lemma}


\section{Intersection theory on strictly semi-stable schemes}

Let $R$ be a discrete valuation ring or a field and let $\pi:\mathcal X\to \Spec R$ be a separated scheme of finite type over $R$. 
Let further $\gamma\in Z^a(\mathcal X)$ and $\gamma'\in Z^b(\mathcal X)$ be cycles whose (set-theoretic) intersection lies in the smooth locus $U\subset \mathcal X$  of $\pi$.
Then Fulton defines the intersection $[\gamma]\cdot [\gamma']$ as a class in the Chow group of $|\gamma|\cap |\gamma'|$, see \cite[Section 20.2]{fulton}.
In particular, if $\gamma$ and $\gamma'$ meet properly, then the intersection $\gamma\cdot \gamma'$  is defined on the level of cycles and Fulton proves that the corresponding class in the Chow group of $U$ depends only on the rational equivalence classes of $\gamma$ and $\gamma'$ on $U$.

A cycle $\gamma$
on $\mathcal{X}$ is said to be flat over $R$ if its support
(as a reduced scheme) is flat over $R$. 
This is equivalent
to saying that no component of $\gamma$ is contained in the special fibre. 

Recall that in general there is no intersection product on the Chow group of singular varieties.
For us, it will  however be important to compute the intersection of certain cycles on strictly semi-stable $R$-schemes and in particular on their special fibres, which are not necessarily smooth and it will not be enough for us to know the corresponding identities on the smooth locus.
In the following two subsections, we deduce two auxiliary results from Fulton's theory, which allow us to overcome this difficulty.

\subsection{Compatibility of specialization and products in a non-smooth setting}

Let $R$ be a discrete valuation ring and let $\pi:\mathcal X\to \Spec R$ be a 
flat morphism. 
Taking the closure of a cycle on $X_\eta$ and restricting that cycle to $X_0$ yields a well-defined specialization map
$$
\spe:\CH^\ast(X_\eta)\to \CH^\ast(X_0),
$$
see \cite[Section 20.3]{fulton}.
If $\pi$ is smooth, this specialization map commutes with the product structure on both sides.
In the following we discuss a variant of this compatibility result in a situation where $\pi$ is not necessarily smooth.

Let $\gamma\in Z^a(\mathcal X)$ and $\gamma'\in Z^{b}(\mathcal X)$ be cycles which meet properly in the smooth locus of $\pi$ and such that  $|\gamma|\cap |\gamma'|$ is flat over $R$. 
This implies that the cycles $\gamma_0$ and $\gamma_0'$ on $X_0$ meet properly in the smooth locus of $X_0$.
Hence, their intersection
$$
\gamma_0\cdot \gamma_0' \in Z^{a+b}(X_0)
$$
is defined on the level of cycles. 

\begin{lemma} \label{lem:specialization}
In the above notation, we have
$$
\spe([\gamma_\eta] \cdot [\gamma'_\eta])= [\gamma_0 \cdot \gamma'_0] \in \CH^{a+b}(X_0).
$$  
\end{lemma}

\begin{proof}
Since $\gamma$ and $\gamma'$ meet properly in the smooth locus of $\pi$, we can define their intersection on the level of cycles
$$
\gamma\cdot \gamma'\in Z^{a+b}(\mathcal X) .
$$
By the construction of the specialization homomorphism on Chow groups, 
\begin{align} \label{eq:sp}
\spe([\gamma_\eta] \cdot [\gamma'_\eta])= [(\gamma\cdot \gamma')_0],
\end{align}
where $(\gamma\cdot \gamma')_0\in Z^{a+b}(X_0)$ denotes the restriction of the cycle $\gamma\cdot \gamma'$ to $X_0$.

We claim that
\begin{align} \label{eq:lem}
(\gamma\cdot \gamma')_0= \gamma_0 \cdot \gamma'_0 \in Z^{a+b}(X_0) .
\end{align}
To see this, note first that $|\gamma_0|\cap |\gamma_0'|$ is contained in the smooth locus of $X_0$.
Recall next that for any smooth $R$-scheme $U$, and cycles $\alpha$ and $\beta$ on $U$, Fulton defines an intersection product $\alpha\cdot \beta$ as a class in the Chow group of $|\alpha|\cap | \beta|$ and this construction commutes with specialization, see \cite[Section 20.3]{fulton}.
Applying this to the smooth locus $U\subset \mathcal X$ of $\pi$, we find that 
$$
[(\gamma\cdot \gamma')_0]= [\gamma_0] \cdot [\gamma'_0] \in \CH^{\ast} (|\gamma_0|\cap |\gamma_0'|).
$$
Since $\gamma_0$ and $\gamma_0'$ meet properly, the above equality has to hold already on the level of cycles and so (\ref{eq:lem}) follows.

The lemma follows now from (\ref{eq:sp}) and (\ref{eq:lem}).
\end{proof}

\subsection{Intersection theory on the special fibre} 
Let $R$ be a discrete valuation ring 
and let $\pi:\mathcal X\to \Spec R$ be a strictly semi-stable $R$-scheme.
Let us denote by $X_{0i},X_{0j}\subset X_0$ irreducible components of the special fibre of $\pi$ and assume that $X_{0ij}:=X_{0i}\cap X_{0j}$ is non-empty.
Consider the following diagram, where all morphisms are given by the natural inclusions:
$$
\xymatrix{&X_{0i} \ar[rd]^{\epsilon_i} \ar[r]^{\epsilon_i^0} & X_0 \ar[d]^{\iota} \\
X_{0ij} \ar[ru]^{\epsilon^i_{ij}} \ar[rd]_{\epsilon^j_{ij}} \ar[rr]^{\epsilon_{ij}} &         &\mathcal X \\
 & X_{0j} \ar[ru]_{\epsilon_j} \ar[r]_{\epsilon_j^0} &  X_0 \ar[u]_{\iota}}
$$

Let $\beta\in Z^{b}(\mathcal X)$ be a cycle that is flat over $R$.
Since $\mathcal X$ is semi-stable over $R$, $X_{0i}$ and $X_{0j}$ are Cartier divisors on $\mathcal X$, and $X_{0ij}$ is a Cartier divisor on $X_{0i}$, as well as on $X_{0j}$.
Hence, $\epsilon_{ij}$, $\epsilon_i$ and $\epsilon_j$ are regular embeddings and so the pullbacks
$$
(\epsilon_{ij})^\ast [\beta]\in \CH^b(X_{0ij}),\ \ (\epsilon_i)^\ast [\beta] \in \CH^b(X_{0i}) \ \ \text{and}\ \ (\epsilon_j)^\ast[\beta]\in \CH^b(X_{0j})
$$
are well-defined on the level of Chow groups, see \cite[Section 20.1]{fulton}.
For any closed subvariety $Z\subset X_{ij}$, the projection formula thus shows
$$
(\epsilon_{ij}^i)_\ast ([Z]\cdot (\epsilon_{ij})^\ast[\beta])=(\epsilon_{ij}^i)_\ast [Z]\cdot (\epsilon_{i})^\ast[\beta] \ \ \text{and}\ \ 
(\epsilon_{ij}^j)_\ast ( [Z]\cdot (\epsilon_{ij})^\ast[\beta])=(\epsilon_{ij}^j)_\ast [Z] \cdot (\epsilon_{j})^\ast[\beta] .
$$
Applying the pushforwards $(\epsilon^0_i)_\ast$ and $(\epsilon^0_j)_\ast$, respectively, this implies 
\begin{align} \label{eq:epsilon}
(\epsilon^0_i)_\ast ((\epsilon_{ij}^i)_\ast [Z]\cdot (\epsilon_{i})^\ast[\beta])
= (\epsilon^0_j)_\ast ((\epsilon_{ij}^j)_\ast [Z] \cdot (\epsilon_{j})^\ast[\beta]) \in \CH^\ast (X_0) .
\end{align}
We use this compatibility, to prove the following.

\begin{lemma} \label{lem:intersection:snc-divisor}
Let $R$ be a discrete valuation ring 
and let $\pi:\mathcal X\to \Spec R$ be a strictly semi-stable $R$-scheme with special fibre $X_0$.
Let $\alpha_0,\tilde \alpha_0\in Z^a(X_0)$ be cycles on $X_0$ that are supported on the smooth locus of $X_0$.
Let further $\beta\in Z^b(\mathcal X)$ be a cycle on $\mathcal X$ that is flat over $R$ and denote its restriction to $X_0$ by $\beta_0 $.
Assume that $\beta_0$ meets $\alpha_0$ and $\tilde \alpha_0$ properly.
If $\alpha_0$ and $\tilde \alpha_0$ are rationally equivalent on $X_0$, then so are the cycles
$$
\alpha_0 \cdot \beta_0\in Z^{a+b}(X_0)\ \ \text{and}\ \ \tilde \alpha_0 \cdot \beta_0 \in Z^{a+b}(X_0) .
$$
\end{lemma}
\begin{proof}
Let $\tau:\tilde X_0 \to X_0$ be the normalization of $X_0$.
Since $\pi$ is strictly semi-stable,
$
\tilde X_0=\sqcup_i \tilde X_{0i}
$
is the disjoint union of the components $ \tilde X_{0i}=X_{0i}$ of $X_0$. 
Let $\tau_i:=\tau|_{\tilde X_{0i}}$ be the restriction of $\tau$ to the component $\tilde X_{0i}$ and let $\iota:X_0\to \mathcal X$ be the inclusion of the special fibre. 
Then $\iota\circ \tau_i$ is a regular embedding and so we may consider the pullback 
$$
(\iota\circ \tau_i)^\ast [\beta] \in \CH^{b}(\tilde X_{0i}). 
$$
Since $\beta$ is flat over $R$, the above pullback is already defined on the level of cycles.
This allows us to define
$$
\tau^\ast \beta_0:=\sum_i (\iota\circ \tau_i)^\ast \beta \in Z^b(\tilde X_0)=\bigoplus_i Z^b(\tilde X_{0i}).
$$

Since $\alpha_0$ and $\tilde \alpha_0$ are supported on the smooth locus of $X_0$, the pullbacks $\tau^\ast \alpha_0,\tau^\ast \tilde \alpha_0 \in Z^a(\tilde X_0)$ are defined on the level of cycles as well. 
By construction of these pullbacks, we have the following identities of cycles on $X_0$:
\begin{align} \label{eq:alpha.beta}
\tau_\ast(\tau^\ast \alpha_0 \cdot \tau^\ast \beta_0)= \alpha_0 \cdot \beta_0
\ \ \text{and} \ \ 
\tau_\ast(\tau^\ast \tilde \alpha_0 \cdot \tau^\ast \beta_0)= \tilde \alpha_0 \cdot \beta_0.
\end{align} 

Assume now that $\alpha_0$ and $\tilde \alpha_0$ are rationally equivalent on $X_0$.
To show that the intersections $\alpha_0 \cdot \beta_0$ and $\tilde \alpha_0 \cdot \beta_0$ are rationally equivalent on $X_0$ as well, it suffices by (\ref{eq:alpha.beta})  to prove the following:
Let $Z\subset \tilde X_{0}$ be a closed integral subscheme of codimension $a$, then the class 
\begin{align} \label{eq:Z.beta}
\tau_{\ast}([Z]\cdot [ \tau^\ast \beta_0]) \in \CH^{a+b}(X_0)
\end{align}
is invariant under the following operations:
\begin{enumerate}
\item replace $Z$ by a rationally equivalent cycle on $\tilde X_0$; \label{item:1}
\item if $Z\subset \tilde X_{0i}$ and $\tau(Z)\subset X_{0i}\cap X_{0j}$, then replace $Z$ by the unique subscheme $Z'\subset \tilde X_{0j}$ with $\tau(Z)=Z'$.  \label{item:2}
\end{enumerate}
Since $\tilde X_0$ is smooth, it is clear that (\ref{eq:Z.beta}) is invariant under (\ref{item:1}), see \cite[\S 1.4 and \S 8]{fulton}.
Moreover, invariance under (\ref{item:2}) follows from (\ref{eq:epsilon}), which concludes the proof of the lemma.
\end{proof}

\section{The diagonal distinguishes stable birational types}

It is well-known that the non-existence of a decomposition of the diagonal prevents a variety from being stably rational, \emph{i.e.}\ from being stably birational to a point. 
The purpose of this section is to prove the following theorem, which shows that decompositions of the diagonal allow to distinguish two stable birational types in much greater generality. 
Partial results in this direction have previously been obtained by Voisin \cite{Voi19}.

\begin{theorem} \label{thm:degeneration:mixed:char:1}
Let $R$ be a discrete valuation ring with algebraically closed residue field $k$.
Let $\pi:\mathcal X\to \Spec R$ and $\pi':\mathcal X'\to \Spec R$ be flat projective morphisms with geometrically connected fibres such that $\pi$ is strictly semi-stable and $\pi'$ is smooth.
Assume 
\begin{enumerate}
\item the special fibre $X_0$ of $\pi$ has universally trivial Chow group of zero-cycles;
\item the special fibre $X'_0$ of $\pi'$ does not admit a decomposition of the diagonal.
\end{enumerate}
Then the geometric generic fibres of $\pi$ and $\pi'$ are not stably birational to each other.
\end{theorem}

If $\pi$ is smooth, then the above theorem is known and follows by similar arguments as in \cite{Voi19}.
However, for most applications (e.g.\ to hypersurfaces) it is essential to allow  $\pi$ to have singular special fibre.

Theorem \ref{thm:degeneration:mixed:char:1}  
can be seen as a cycle-theoretic analogue of the main result in \cite{NS}, where Nicaise and Shinder showed that the geometric generic fibres of two strictly semi-stable families over a discrete valuation ring of equal characteristic zero are not stably birational, unless the special fibres have the same class in the Grothendieck ring of varieties modulo the class of the affine line. 

Compared to \cite{NS,Sh19}, the main advantage of the above theorem is that it works over fields of arbitrary characteristic.
Note however that there is also an advantage in characteristic zero, because the assumption on the universal $\CH_0$-triviality 
holds in various cases
where Shinder's assumption does not. 
For instance, a general arrangement of $d$ hyperplanes in $\CP^{n+1}$ is always universally $\CH_0$-trivial, but its class in the Grothendieck group of varieties is congruent to $1$ mod $\LL$ only if $d\leq n+1$.

As an immediate corollary of Theorem \ref{thm:degeneration:mixed:char:1} and \cite[Lemma 8]{Sch19}, 
we have for instance.

\begin{corollary} \label{cor:degeneration:main}
Let $k$ be an uncountable algebraically closed field of arbitrary characteristic.
Let $\pi:\mathcal X\to C$ be a flat projective morphism between smooth $k$-varieties $\mathcal X$ and $C$ with $\dim C=1$.
Let $0\in C$ be a closed point and assume that
\begin{enumerate}
\item the very general fibre of $\pi$ does not admit a decomposition of the diagonal;
\item the special fibre $X_0$ is a reduced simple normal crossing divisor on $\mathcal X$ which has universally trivial Chow group of zero-cycles.
\end{enumerate}
Then two very general fibres of $\pi$ are not stably birational to each other.
\end{corollary}

The key step in the proof of Theorem \ref{thm:degeneration:mixed:char:1} is the following proposition.

\begin{proposition} \label{prop:degeneration:mixed:char}
Let $R$ be a discrete valuation ring with algebraically closed residue field $k$.
Let $\pi:\mathcal X\to \Spec R$ and $\pi':\mathcal X'\to \Spec R$ be flat projective schemes  with connected fibres such that $\pi$ is strictly semi-stable and $\pi'$ is smooth.
Assume 
\begin{enumerate}
\item the generic fibres of $\pi$ and $\pi'$ are stably birational to each other;
\item the special fibre $X_0$ of $\pi$ has universally trivial Chow group of zero-cycles.
\end{enumerate}
Then the special fibre $X_0'$ of $\pi'$ admits a decomposition of the diagonal.
\end{proposition}


\begin{proof}[Proof of Proposition \ref{prop:degeneration:mixed:char}] 
Up to multiplying $\mathcal X$ and $\mathcal X'$ with some projective spaces (possibly of different dimensions), we may assume that there is a birational map $f:X_\eta\dashrightarrow X'_\eta$.
In particular, $\pi$ and $\pi'$ are of the same relative dimension, which we denote by $n$.
Let 
$$
\Gamma\subset \mathcal X\times_R \mathcal X'
$$
be the closure of the graph of $f$, which is a cycle of codimension $n$ on $\mathcal X\times_R \mathcal X'$.
Note that $\Gamma$ is automatically flat over $R$. 

To explain the idea of the proof, assume first that $\pi$ is smooth and let $ \Gamma^{t}\subset  \mathcal X'\times_R \mathcal X$ be the transpose of $\Gamma$, which is nothing but the closure of the graph of $f^{-1}$.
Since $\pi$ is smooth, $\mathcal X\times_R \mathcal X'$ is smooth over $R$ and so we can define the composition of cycles $\Gamma\circ\Gamma^t\in \CH^{n}(\mathcal X'\times_R\mathcal X')$ as pushforward of 
\begin{align} \label{eq:explanation}
(\pr_{12}^\ast \Gamma^t)\cdot (\pr_{23}^\ast \Gamma) \in \CH^n(\mathcal X'\times _R\mathcal X\times_R\mathcal X').
\end{align}
Since $f\circ f^{-1}=\id$, the restriction of this cycle to the generic fibre of $\mathcal X'\times_R \mathcal X' \to \Spec R$ is the diagonal.
It thus follows from the specialization map on Chow groups \cite{fulton} that the special fibre of the above cycle must be rationally equivalent to the diagonal of $X'_0$.
On the other hand, this cycle is given by the pushforward of
\begin{align}  \label{eq:explanation2}
(\pr_{12}^\ast \Gamma^t_0)\cdot (\pr_{23}^\ast \Gamma_0) \in \CH^n(  X'_0\times _k X_0\times_k X'_0) 
\end{align}
to $X'_0\times_k X'_0$.
Since $X_0$ has universally trivial Chow group of zero-cycles, the restriction of $\Gamma^t_0$ to the generic fibre $ k(X'_0)\times_k X_0$ of the projection $X'_0\times_kX_0\to X'_0$ is rationally equivalent to the base change $k(X'_0)\times z$ of a zero-cycle $z$ on $X_0$.
As we already know that the pushforward of (\ref{eq:explanation2}) to $X'_0\times_k X'_0$ is rationally equivalent to the diagonal, this information is enough to conclude that  $X'_0$ admits a decomposition of the diagonal.
In this argument, we used heavily that $\pi$ is smooth, as the intersections in (\ref{eq:explanation}) and (\ref{eq:explanation2}) are not even defined in the non-smooth setting.
The idea that allows us to bypass this difficulty is to replace in (\ref{eq:explanation}) the first factor of $\mathcal X'\times _R\mathcal X\times_R\mathcal X'$ by the localization of $\mathcal X'$ at the generic point of $X'_0\subset \mathcal X'$.
This way, the intersections in (\ref{eq:explanation}) and (\ref{eq:explanation2}) are of relative dimension zero and so we can hope to use a moving lemma to make sense of these intersections.
Moreover, by localizing at the generic point of $X'_0\subset \mathcal X'$ it is still possible to exploit the information that $ k(X_0')\times X_0$ has trivial Chow group, which will eventually allow us to prove what we want. 
We give the details of this argument in what follows.

Let $A:=\mathcal O_{\mathcal X',\eta_{X'_0}}$ be the local ring of $\mathcal X'$ at the generic point of $X'_0\subset \mathcal X'$. 
This is a discrete valuation ring, because $X'_0\subset \mathcal X'$ is a Cartier divisor.
By Lemma \ref{lem:ss}, the base change $\mathcal X_A:=\mathcal X\times_R\Spec A  $ is strictly semi-stable over $A$. 

Let 
$$
\tilde \gamma\in Z^n(\mathcal X_A)
$$ 
be the relative zero-cycle that is induced by $\Gamma$.
Since $\pi$ is strictly semi-stable, the smooth locus $U\subset \mathcal X$ of $\pi$ meets each component of any fibre of $\pi$ in a non-empty Zariski open subset.
Since $\mathcal X_A$ is regular, the moving lemma for relative zero-cycles in \cite[Theorem 2.3]{GLL} thus shows that 
\begin{enumerate}
\item[($\sharp$)]
there is a cycle $\gamma\in Z^n(\mathcal X_A)$ with $[\gamma]=[\tilde \gamma]\in \CH^n(\mathcal X_A)$, which does not meet the singular locus of $\pi_A:\mathcal X_A\to \Spec A$ and also not the locus in $\mathcal X_A$ over which the projection $\pr_1:\Gamma_A\to \mathcal X_A$ is non-flat, where $\Gamma_A\subset \mathcal X_A\times_A\mathcal X'_A$ denotes the base change of $\Gamma$ to $A$.
\end{enumerate}

By Lemma \ref{lem:ss}, $\mathcal X_A\times_A\mathcal X'_A $ is a semi-stable $A$-scheme.
Consider the following diagram
$$
\xymatrix{
& \mathcal X_A\times_A\mathcal X'_A \ar[rd]^{p}  \ar[ld]_{q}&\\
\mathcal X_A  & & \mathcal X'_A }
$$
where $p$ and $q$ denote the corresponding projections.
Since $q$ is flat, there is a flat pullback map $q^\ast$ on the level of cycles.
Condition ($\sharp$) implies the following:
\begin{enumerate} 
\item[($\sharp'$)] 
$q^\ast \gamma$ is a cycle on $ \mathcal X_A\times_A\mathcal X'_A $ which intersects $\Gamma_A$ properly in the smooth locus of  $ \mathcal X_A\times_A\mathcal X'_A $ over $A$. 
\end{enumerate} 
Therefore, the intersection $(q^\ast\gamma) \cdot (\Gamma_A)$  is defined on the level of cycles, and so we may consider the cycle
\begin{align} \label{eq:alpha}
\alpha:= p_\ast( (q^\ast\gamma) \cdot (\Gamma_A)) \in Z^n( \mathcal X'_A).
\end{align} 
By ($\sharp$), the generic fibre of this cycle is rationally equivalent to 
$$
p_\ast( (q^\ast \tilde \gamma_\eta) \cdot (\Gamma_\eta)_{F(X'_\eta)}) \in Z^n\left( (X'_\eta)_{F(X'_\eta)}\right),
$$ 
where $F$ denotes the fraction field of $R$.
Since $\Gamma$ and $\tilde \gamma$ are induced by the birational map $f:X_\eta\dashrightarrow X'_\eta$, the above cycle coincides with the zero-cycle $\delta_{X'_\eta}\in Z^n\left((X'_\eta)_{F(X'_\eta)}\right)$ that is induced by the diagonal of $X'_\eta$.
Hence,
$$
[\alpha_\eta]=[\delta_{X'_\eta}] \in \CH^n\left((X'_\eta)_{F(X'_\eta)}\right) .
$$
Applying the specialization map on Chow groups, we thus find
\begin{align}  \label{eq:alpha_0=delta}
[\alpha_0]=[\delta_{X'_0}] \in \CH^n\left((X'_0)_{K}\right) ,
\end{align}
where $K=k(X'_0)$ denotes the residue field of $A$ and $\delta_{X'_0}$ denotes the zero-cycle induced by the diagonal of $X'_0$.
In what follows, we use Lemmas \ref{lem:specialization} and \ref{lem:intersection:snc-divisor} to compute $[\alpha_0]$ in another way, which will allow us to conclude that $X'_0$ admits a decomposition of the diagonal.

By ($\sharp'$), Lemma \ref{lem:specialization} applies and we get
$$
\spe([q^\ast\gamma_\eta] \cdot [\Gamma_{F(X'_\eta)}])=[(q^\ast\gamma_0) \cdot (\Gamma_0)_K] \in \CH_{0}\left((X_0)_{K}\times_{K} (X_0')_{K}\right),
$$
where $q^\ast\gamma_0$ is the flat pullback of the zero-cycle $\gamma_0\in Z_0((X_0)_{K})$, given as special fibre of $\gamma \to \Spec A$ and $(\Gamma_0)_K$ is obtained via base change to $K$ of the $n$-cycle on $X_0\times X'_0$ given by the special fibre of $\Gamma \to \Spec R$.
Since specialization and proper pushforward commute (see \cite[Proposition 20.3]{fulton}), this yields
\begin{align} \label{eq:alpha_0:1}
[\alpha_0]=(p_0)_\ast\left( [(q_0^\ast\gamma_0) \cdot (\Gamma_0)_{K}]\right) \in \CH_0\left((X_0')_K\right).
\end{align}

Since $X_0$ has universally trivial Chow group of zero-cycles by assumption, 
$$
[\gamma_0]=[z_K]\in \CH^n\left((X_0)_K\right),
$$
where $z\in Z^n(X_0)$ is a zero-cycle with $\deg z=\deg \gamma_0=\deg \gamma_\eta=\deg \tilde \gamma_\eta =1$. 
By Lemma \ref{lem:intersection:snc-divisor}, we conclude from (\ref{eq:alpha_0:1}) that
$$
[\alpha_0]=(p_0)_\ast\left( [ (z\times X'_0)_K \cdot (\Gamma_0)_{K}]\right) 
$$ 
holds in $\CH_0( (X_0')_{K} ) $. 
This implies in particular 
$$
[\alpha_0]\in \im\left( \CH_0(X'_0)\to \CH_0\left((X'_0)_{K}\right)\right).
$$ 
Comparing this with (\ref{eq:alpha_0=delta}), we conclude that $X'_0$ admits a decomposition of the diagonal, as claimed.
This concludes the proof of the proposition.
\end{proof}

Proposition \ref{prop:degeneration:mixed:char} deals with stably birational generic fibres.
To deal with stably birational geometric generic fibres, we need the following result of Hartl \cite{hartl} about the behaviour of strictly semi-stable schemes under base change. 

\begin{lemma} \label{lem:hartl}
Let $R$ be a discrete valuation ring with algebraically closed residue field $k$ and let $\pi:\mathcal X\to \Spec R$ be strictly semi-stable.
Let $R\subset R'$ be a finite extension of discrete valuation rings and let $\mathcal X_{R'}$ be the base change of $\mathcal X$ to $R'$.
Then there is a finite sequence of blow-ups
$$
\tilde{\mathcal X}:=V^r\to V^{r-1}\to \dots \to V^1\to V^0:=\mathcal X_{R'} , 
$$ 
where $\tilde{\mathcal X}$ is strictly semi-stable over $R'$ and $V^i\to V^{i-1}$ is the blow-up of $V^{i-1}$ along an irreducible component of the special fibre $V^i_0$ of $V^{i}$ which is not Cartier.
Moreover, pushforward via the natural map induces an isomorphism on Chow groups
$$
\CH_0\left((\tilde X_0)_L\right)\stackrel{\sim}\longrightarrow \CH_0\left((X_0)_L\right)
$$
after base change to any extension $L$ of $k$.
\end{lemma}

\begin{proof}
The first assertion is \cite[Proposition 2.2]{hartl}.
The proof of \cite[Proposition 2.2]{hartl} shows that in each step $V^i\to V^{i-1}$, the central fibre $V^i_0$ is obtained from $V^{i-1}_0$ by glueing in a Zariski locally trivial $\CP^1$-bundle.
Hence the assertion about the Chow group of $\tilde X_0$ follows.
\end{proof}

We are now able to prove Theorem \ref{thm:degeneration:mixed:char:1}.

\begin{proof}[Proof of Theorem \ref{thm:degeneration:mixed:char:1}] 
Replacing $R$ by its completion, we may assume that $R$ is a complete discrete valuation ring and we denote the fraction field of $R$ by $K$. 
For a contradiction, we assume that $X_{\overline \eta}$ is stably birational to $X'_{\overline \eta}$.
Then there is a finite field extension $L/K$, such that $X_{\eta}\times_KL$ is stably birational to $X'_{ \eta}\times_KL$.
Let $R'$ denote the integral closure of $R$ in $L$.
Since $R$ is a complete discrete valuation ring,  $R\to R'$ is a finite map of rings and $R'$ is a complete discrete valuation ring as well, see \cite[Th\'eor\`eme 23.1.5 and Corollaire 23.1.6]{EGAIV}. 
Replacing $\pi$ and $\pi'$ by the base change to $R'$, our families have all the initial properties, apart from the strict semi-stability of $\mathcal X$, which we will lose in general.
However, by Lemma \ref{lem:hartl}, there is projective birational morphism $\tilde{\mathcal X}\to \mathcal X$ such that $\tilde{\mathcal X}$ is a strictly semi-stable $R$-scheme and such that the special fibres $\tilde X_0$ and $X_0$ have isomorphic Chow groups of zero-cycles after any extension of the base field.
Hence, $\tilde X_0$ has universally trivial Chow group of zero-cycles, because the same holds true for $X_0$ by assumptions.
Therefore, Proposition \ref{prop:degeneration:mixed:char} shows that $X'_0$ admits a decomposition of the diagonal, which contradicts our assumptions.
This proves the theorem.
\end{proof} 

\vspace{-0.3cm}\clearpage

\section{Applications to hypersurfaces} 

Theorem \ref{thm:hypersurface:2} follows from the following more precise result.

\begin{theorem} \label{thm:hypersurface:3}
Let $k$ be an uncountable algebraically closed field of arbitrary characteristic and let
$Y$ be a smooth projective variety over $k$ which does not admit a decomposition of the diagonal.
For some $d,n\geq 1$, let $\ell\subset \CP(H^0(\CP^{n+1}_k,\mathcal O(d)))$ be a general pencil of hypersurfaces of degree $d$ in $\CP^{n+1}_k$ which contains a general arrangement of $d$ hyperplanes as one of its members.
(The pencil $\ell$ does not need to be very general, and also not general with respect to $Y$.)

Then $Y$ is not stably birational to a very general member of the pencil $\ell$. 
\end{theorem}

\begin{proof}
Let $\pi:\mathcal X\to \ell\cong \CP^1_k$ be the universal family of the pencil.
By assumptions, there is a point $0\in \ell$ such that $X_0=\{l_1\cdots l_d=0\}$ for general linear polynomials $l_i$. 
By the genericness assumption on $\ell$,
$\mathcal X$ is smooth away from the central fibre $X_0$.
In \cite[Lemma 3.6]{Sh19}, Shinder computes a resolution of singularities $\tau:\widetilde {\mathcal X}\to \mathcal X$ by repeatedly blowing-up the proper transforms of $\{l_i=0\}$ for $i=1,\dots ,d$ and shows that $\tilde X_0$ is a reduced simple normal crossing divisor on $\tilde{\mathcal X}$, all of whose components are rational. (Shinder's paper is written over an uncountable algebraically closed field of characteristic zero, but that computation holds more generally over an arbitrary field.) Hence, $\tilde X_0$ has universally trivial Chow group of zero-cycles.

Consider the families $\tilde \pi:\widetilde {\mathcal X}\to \ell$ and $\pi':\mathcal X':=Y\times_k \ell\to \ell$.
Applying Theorem \ref{thm:degeneration:mixed:char:1} to the base change of these families  to the local ring of $\ell$ at $0$, we conclude that the geometric generic fibres $\widetilde X_{\overline \eta}$ and $Y\times \overline {k(\ell)}$ are not stably birational to each other. 
We thus conclude from \cite[Lemma 8]{Sch19} that also the very general fibres of $\tilde \pi$ and $\pi':\mathcal X'\to \ell$ are not stably birational to each other.
That is, the very general fibre of $\tilde \pi$ (which coincides with the very general fibre of $\pi$) is not stably birational to $Y$, as we want.
\end{proof}


\begin{proof}[Proof of Corollary \ref{cor:hypersurface}]
Assume that there is a hypersurface $Y\subset \CP^{n+1}_k$ of degree $d$ which does not admit a decomposition of the diagonal.
By the specialization map on Chow groups \cite[Proposition 11.1]{fulton},  we may assume that $Y$ is very general and so it is in particular smooth.
Hence, Corollary \ref{cor:hypersurface} follows from Theorem~\ref{thm:hypersurface:2}.
\end{proof}

\begin{proof}[Proof of Corollary \ref{cor:hypersurface:char_neq_2}]
Passing from $k$ to its algebraic closure, we may assume that $k$ is an uncountable algebraically closed field of characteristic different from two.
By \cite[Theorem 8.1]{Sch18}, a very general hypersurface $X\subset \CP_k^{n+1}$ of degree $d\geq \log_2(n)+2$ and dimension $n\geq 3$ does not admit a decomposition of the diagonal and so the result follows from Corollary \ref{cor:hypersurface}.
\end{proof}

\begin{proof}[Proof of Corollary \ref{cor:hypersurface:char=2}]
Passing from $k$ to its algebraic closure, we may assume that $k$ is an uncountable algebraically closed field of characteristic two. 
By \cite[Corollary 1.2]{Sch-torsion}, a very general hypersurface $X\subset \CP_k^{n+1}$ of degree $d\geq \log_2(n)+3$ does not admit a decomposition of the diagonal and so the result follows from Corollary \ref{cor:hypersurface}.
\end{proof}


\begin{theorem} \label{thm:ex:Q(t)}
Let $k=\Q(t)$ or $k=\F_p(s,t)$.
There are smooth Fano hypersurfaces $X$ and $X'$ over $k$ of the same dimension and degree, which over the algebraic closure of $k$ are neither stably rational, nor stably birational to each other.
Moreover, if $k$ is not of characteristic two, the slopes $\frac{\deg X}{\dim X+1}$ of these examples may be chosen arbitrarily small.
\end{theorem}

\begin{proof}
We may write $k=F(t)$, where $F=\Q$ or $F=\F_p(s)$.
Let us first deal with the case where $k$ has characteristic different from two.
Then \cite[Theorem 8.3]{Sch18} implies that there is a smooth projective Fano hypersurface $Z$ over $F$ of arbitrarily small slope whose base change to the algebraic closure $\overline F$ of $F$ does not admit a decomposition of the diagonal.
We may write $Z=\{f=0\}$ for some irreducible homogeneous polynomial $f\in F[x_0,\dots ,x_{n+1}]$ of degree $d$. (It is possible to deduce an explicit description of $f$ from \cite{Sch18}.)

Consider the discrete valuation ring $R:=\overline{F}[t]_{(t)}$ whose residue field is the algebraic closure of $F$. 
We then define the flat projective $R$-schemes
$$
\mathcal X:=\{tf+l_1\cdots l_d=0\} \ \ \text{and}\ \ \mathcal X':= \{f+tg=0\},
$$
where $g\in F[x_0,\dots ,x_{n+1}]$ is general with $\deg g=d$ and $l_1\cdots l_d$ is a product of general linear polynomials $l_i\in F[x_0,\dots ,x_{n+1}]$.
Note that the generic fibres $X_\eta$ and $X'_\eta$ can be defined over $k$.
Moreover, the geometric generic fibres $X_{\overline \eta}$ and $X'_{\overline \eta}$ are smooth and do not admit decompositions of the diagonal, because they both specialize (via $t^{-1}\to 0$, resp.\ $t\to 0$) to $Z_{\overline F}$, which is smooth and does not admit a decomposition of the diagonal.
In particular, $X_{\overline \eta}$ and $X'_{\overline \eta}$ are stably irrational; it remains to show that they are not stably birational to each other.

As in the proof of Theorem \ref{thm:hypersurface:2}, Shinder's computation in \cite[Lemma 3.6]{Sh19} shows that there is a projective modification $\tilde {\mathcal X}\to \mathcal X$, such that $\widetilde{\mathcal X}$ is strictly semi-stable over $R$ and such that $\tilde X_0$ has universally trivial Chow group of zero-cycles.
Since $g$ is general, $\mathcal X'$ is smooth over $R$.
Moreover, $X'_0\cong Z_{\overline F}$ does not admit a decomposition of the diagonal.
Applying Theorem \ref{thm:degeneration:mixed:char:1}, we thus find that $X_{\overline \eta}$ and $X'_{\overline \eta}$ are not stably birational to each other.
This concludes the case where $k$ has characteristic different from two.

If $\operatorname{char} k=2$, then we may argue similarly by using Totaro's result \cite{totaro}, which produces a smooth projective hypersurface $Z\subset \CP^{n+1}_{\F_2(s)}$ of any even degree $d\geq 2\lceil \frac{n+2}{3}\rceil$ which does not admit a decomposition of the diagonal over the algebraic closure of $\F_2(s)$.
The rest of the argument is analogous to the one given above. 
\end{proof}


\section{Questions}

Let $H_{d,n}$ denote the coarse moduli space of smooth hypersurfaces of degree $d$ in $\CP^{n+1}$.
Stable birational equivalence induces an equivalence relation $\sim$ on the set of $k$-rational points $H_{d,n}(k)$ of this moduli space and we can consider the quotient $H_{d,n}(k)/\sim$.
A priori, this quotient is only a set, but we can nonetheless study its size in various ways.
For instance, we may define the dimension of $H_{d,n}(k)/\sim$ as 
$$
\min\{\dim Z\mid Z\subset H_{d,n}(k)\ \text{is closed and}\ Z\to H_{d,n}(k)/\sim\ \text{is surjective}\}.
$$

By Theorem \ref{thm:hypersurface:3}, we know that for any uncountable algebraically closed field $k$ such that the very general element of $H_{d,n}(k)$ does not admit a decomposition of the diagonal, the set $H_{d,n}(k)/\sim$ is uncountable. (By \cite{Sch18}, this holds for instance if $d\geq \log_2(n) +2$, $n\geq 3$ and $\operatorname{char}(k)\neq 2$.)
In particular, $H_{d,n}(k)/\sim$ is positive-dimensional in these cases and it is natural to ask to improve this result.

\begin{question}
What can we say about the dimension of $H_{d,n}(k)/\sim$?
\end{question}

A related question that goes back to Voisin asks about the dimension of the fibres of the quotient map $H_{d,n}(k)\to H_{d,n}(k)/\sim$.

\begin{question}[Voisin]
What is the maximal dimension of a subvariety $Z\subset H_{d,n}(k)$ 
such that all hypersurfaces parametrized by $Z$ are stably birational to each other?
\end{question}

In light of Theorem \ref{thm:ex:Q(t)}, it is also worth to ask the following.

\begin{question}
Are all smooth Fano hypersurfaces of given dimension and degree over $\Q$ stably birational to each other over $\overline \Q$?
\end{question}

\section*{Acknowledgements}   
I am grateful to the excellent referees for  corrections and comments which significantly improved the exposition, and to C.\ Liedtke for pointing out the reference \cite{hartl}.


\clearpage

\vspace{0.5cm}

\bibliographymark{References}

\providecommand{\bysame}{\leavevmode\hbox to3em{\hrulefill}\thinspace}
\providecommand{\arXiv}[2][]{\href{https://arxiv.org/abs/#2}{arXiv:#1#2}}
\providecommand{\MR}{\relax\ifhmode\unskip\space\fi MR }
\providecommand{\MRhref}[2]{%
  \href{http://www.ams.org/mathscinet-getitem?mr=#1}{#2}
}
\providecommand{\href}[2]{#2}

\end{document}